\documentclass[10pt]{amsart}

\usepackage{amsmath, mathtools, amsthm, amscd, amssymb, amsfonts, pstricks,bbm, verbatim, mathrsfs}
\usepackage[super]{nth}

\newcommand{\R}{{\mathbb R}}

\newcommand{\1}{\mathbbm{1}}

\DeclareMathOperator{\dist}{dist}
\DeclareMathOperator{\rank}{rank}

\DeclareMathOperator{\pos}{pos}
\DeclareMathOperator{\expan}{expansion}
\DeclareMathOperator{\cont}{contraction}

\newtheorem{theorem}{Theorem}[section]
\newtheorem{lemma}[theorem]{Lemma}
\newtheorem{proposition}[theorem]{Proposition}
\newtheorem{corollary}[theorem]{Corollary}

\newtheorem{definition}[theorem]{Definition}

\numberwithin{equation}{section}

\theoremstyle{definition}

\newtheorem{remark}[theorem]{Remark}

\begin{document}
	
	\title[Negative Type and Bi-lipschitz Embeddings Into Hilbert Space]{Negative Type and Bi-lipschitz Embeddings Into Hilbert Space}
	\author[Gavin Robertson]{Gavin Robertson}
	\email{gavin.robertson@unsw.edu.au}
	\address{School of Mathematics and Statistics, University of New South Wales, Sydney, New South Wales 2052, Australia}
	
	\begin{abstract}
		The usual theory of negative type (and $p$-negative type) is heavily dependent on an embedding result of Schoenberg, which states that a metric space isometrically embeds in some Hilbert space if and only if it has $2$-negative type. A generalisation of this embedding result to the setting of bi-lipschitz embeddings was given by Linial, London and Rabinovich. In this article we use this newer embedding result to define the concept of distorted $p$-negative type and extend much of the known theory of $p$-negative type to the setting of bi-lipschitz embeddings. In particular we show that a metric space $(X,d_{X})$ has $p$-negative type with distortion $C$ ($0\leq p<\infty$, $1\leq C<\infty$) if and only if $(X,d_{X}^{p/2})$ admits a bi-lipschitz embedding into some Hilbert space with distortion at most $C$. Analogues of strict $p$-negative type and polygonal equalities in this new setting are given and systematically studied. Finally, we provide explicit examples of these concepts in the bi-lipschitz setting for the bipartite graphs $K_{m,n}$ and the Hamming cube $H_{n}$.
	\end{abstract}
	\maketitle
	
	\section{Introduction}\label{Sec 1}
	
	The theory of embeddings of metric spaces has a long and rich history, with some of the most classical results in this area dating back to Cayley \cite{Cayley 1} in the \nth{19} century. Since then embeddings of metric spaces have been the subject of much study with classical results in this area given by Menger \cite{Menger 1, Menger 2}, Schoenberg \cite{Schoenberg 1, Schoenberg 2}, Enflo \cite{Enflo 1, Enflo 2} and Gromov \cite{Gromov 1} (to name a few). One main theme that can be seen throughout their work is that many difficult problems in mathematical analysis and otherwise can be solved quickly and easily by constructing a suitable embedding of a metric space into a more tractable host space. Because of this the theory of metric embeddings has a plethora of applications in areas such as functional analysis, evolutionary biology, combinatorial optimisation and theoretical computer science. Specific examples include the Sparsest Cut Problem in combinatorial optimisation \cite{Arora 1, Vazirani 1} and the study of phylogenetic trees in evolutionary biology \cite{Weber 1, Ailon 1, Semple 1}.
	
	In this paper we will study isometric and bi-lipschitz embeddings of metric spaces through the theory of $p$-negative type. The concept of $p$-negative type was originally developed by Schoenberg \cite{Schoenberg 2} to study isometric embeddings of powers of metric spaces into Hilbert space. Recently $p$-negative type has been the topic of much research in areas such as mathematical analysis and theoretical computer science (see \cite{Doust 1, Doust 2, Kelleher 1, Li 1, Faver 1} for a list of some of the applications of $p$-negative type to the isometric theory of metric spaces). The definition of $p$-negative type is as follows (note that the definition of strict $p$-negative type is due to Li and Weston in \cite{Li 1}).
	
	\begin{definition}\label{Negative Type Def}
		Let $(X,d_{X})$ be a semi-metric space\footnote{Here we use the term semi-metric space to mean a `metric' space except that we drop the requirement that the triangle inequality holds. All results in this paper will hold for semi-metric spaces since they do not require the use of the triangle inequality anywhere in their proofs.} and $0\leq p<\infty$. Then $X$ is said to have \textit{$p$-negative type} if
		\[ \sum_{i,j=1}^{n}d_{X}(x_{i},x_{j})^{p}\xi_{i}\xi_{j}\leq 0
		\]
		for all distinct $x_{1},\dots,x_{n}\in X$, $\xi_{1},\dots,\xi_{n}\in\R$ with $\sum_{i=1}^{n}\xi_{i}=0$ and $n\geq 2$. Furthermore, $X$ is said to have strict $p$-negative type if equality holds only for $\xi_{1}=\dots=\xi_{n}=0$.
	\end{definition}
	
	Central to the theory of $p$-negative type and its applications are the two following classical results of Schoenberg \cite{Schoenberg 1, Schoenberg 2} (note that the first statement below provides an isometric characterisation of subsets of Hilbert space).
	
	\begin{theorem}\label{Basic Negative Type Results}
		Let $(X,d_{X})$ be a semi-metric space and $0\leq p<\infty$.
		\begin{enumerate}
			\item $X$ has $p$-negative type if and only if $(X,d_{X}^{p/2})$ isometrically embeds in some Hilbert space.
			\item If $0\leq q<p$ and $X$ has $p$-negative type then $X$ also has strict $q$-negative type.
		\end{enumerate}
	\end{theorem}

	Note that the second statement in the above theorem implies that all of the values of $p$-negative type that a semi-metric space $(X,d_{X})$ possesses are encoded in the largest value $s\geq 0$ such that $X$ has $s$-negative type. Thus one is led to consider the quantity
	\[ \wp_{X}=\sup\{p\geq 0:X\text{ has }p\text{-negative type}\}
	\]
	which is referred to as the supremal $p$-negative type of $X$. Note that since all the sums that appear in the definition of $p$-negative type are finite sums a simple limiting argument shows that if $\wp_{X}<\infty$ then $X$ also has $\wp_{X}$-negative type. Consequently the set of values $s\geq0$ such that a semi-metric space $(X,d_{X})$ has $s$-negative type is always of the form $[0,\wp_{X}]$ or $[0,\infty)$.
	
	It follows from the work of Schoenberg \cite{Schoenberg 1, Schoenberg 2} that if $0<p\leq 2$ then $X=L^{p}$ has $\wp_{X}=p$. In fact, a celebrated theorem of Bretagnolle, Dacunha-Castelle and Krivine \cite{Bretagnolle 1} provides a partial converse to this result, which characterises linear subspaces of $L^{p}$ ($1\leq p\leq 2$) up to linear isometry. Namely, if $1\leq p\leq 2$ then a real normed space $X$ is linearly isometric to a subspace of some $L^{p}$ space if and only if $X$ has $p$-negative type. The case when $2<p<\infty$ is radically different. It follows from Koldobsky \cite{Koldobsky 1} that if $n\geq 3$ and $2<p<\infty$ then $X=\ell^{p}_{n}$ has $\wp_{X}=0$, and also that $Y=\ell^{p}_{2}$ has $\wp_{Y}=1$. The variation for $n=2$ can be explained by the well-known result that every $2$-dimensional normed space is linearly isometric to a subspace of $L_{1}$ (see \cite{Yost 1}). Results on non-commutative $L^{p}$ spaces have also been obtained by Dahma and Lennard \cite{Dahma 1}. Here they showed that if $X=S_{p}$ is the Schatten $p$-trace class then $\wp_{X}=2$ when $p=2$ and $\wp_{X}=0$ if $0<p<\infty$ and $p\neq 2$.
	
	More recently, Linial, London and Rabinovich \cite{Linial 2} have provided a similar characterisation of those finite metric spaces that admit a bi-lipschitz embedding into some Hilbert space with a given level of distortion. Bi-lipschitz embeddings of metric spaces, and finite metric spaces in particular, have been studied extensively over the last half-century or so. Originally such embeddings were studied from the perspective of geometric analysis and Banach space theory. However, more recently the theory of bi-lipschitz embeddings has been examined from the perspective of combinatorial optimisation and theoretical computer science. This is in part due to Linial, London and Rabinovich \cite{Linial 2} who showed that many problems in optimisation could be solved efficiently by considering bi-lipschitz embeddings of certain metric spaces into larger host spaces, such as Hilbert or Banach spaces.
	
	Let us now recall the definition of a bi-lipschitz embedding.
	
	\begin{definition}\label{Bilipschitz Definition}
		Let $(X,d_{X})$ and $(Y,d_{Y})$ be semi-metric spaces and $1\leq C<\infty$.
		\begin{enumerate}
			\item A map $f:X\rightarrow Y$ is said to have \textit{distortion at most $C$} if there exists a (scaling constant) $s>0$ such that
			\[ sd_{X}(x,y)\leq d_{Y}(f(x),f(y))\leq sCd_{X}(x,y)
			\]
			for all $x,y\in X$. The smallest such $C$ for which this holds is denoted by $\dist(f)$ (and if no such $C$ exists we set $\dist(f)=\infty$). If $\dist(f)<\infty$ then we say that $f$ is a bi-lipschitz embedding.
			\item We denote by $c_{(Y,d_{Y})}(X,d_{X})$ (or simply by $c_{Y}(X)$) the infimum of all constants $1\leq C\leq\infty$ such that there exists a map $f:X\rightarrow Y$ with $\dist(f)\leq C$ (where again we allow the possibility that $c_{(Y,d_{Y})}(X,d_{X})=\infty$). When $Y=\ell^{2}$, we denote $c_{Y}(X)$ simply by $c_{2}(X)$.
		\end{enumerate}
	\end{definition}

	\begin{remark}
		For our purposes, it is useful to say that a map $f:(X,d_{X})\rightarrow(Y,d_{Y})$ is an isometry if there exists some $s>0$ such that $d_{Y}(f(x),f(y))=sd_{X}(x,y)$ for all $x,y\in X$. Note that this is slightly more general than the usual definition of isometries in the literature. However, one may note that the defintion of (strict) $p$-negative type is invariant under this more general definition of isometry too. Moreover, with this definition one has that $\dist(f)=1$ if and only if $f$ is an isometry. In this way we may think of bi-lipschitz embeddings as a generalisation of isometries.
	\end{remark}
	
	One of the main problems that one faces when dealing with bi-lipschitz embeddings of semi-metric spaces is how one may embed the given space into the larger space with as little distortion as possible. The most classical result along these lines is perhaps a result of Bourgain \cite{Bourgain 1}, which states that any $n$ point metric space\footnote{The use of the word `metric' is necessary here, since the result fails to hold in general for $n$ point semi-metric spaces.} admits a bi-lipschitz embedding into $\R^{n}$ with distortion at most $O(\log n)$. In \cite{Linial 2} Linial, London and Rabinovich provided an algorithmic proof of this result and were able to show that this bound is attained for constant degree expander graphs (and hence cannot be improved in general).
	
	To state one of the main results of \cite{Linial 2} we need to introduce a certain class of matrices that will frequently appear throughout this paper. We will use $M_{n}(\R)$ to denote the space of all real-valued $n\times n$ matrices. Using $A^{T}$ to denote the usual matrix transpose of $A$ and $\langle\cdot,\cdot\rangle$ to denote the standard inner product on $\R^{n}$, we then set
	\[ M_{n}^{+}(\R)=\{A\in M_{n}(\R):A^{T}=A\mbox{ and }\langle A\xi,\xi\rangle\geq 0,\forall \xi\in\R^{n}\}.
	\]
	That is, $M_{n}^{+}(\R)$ is simply the set of all positive semidefinite matrices in $M_{n}(\R)$. We will also use $\1$ to denote the vector in $\R^{m}$ whose entries are all $1$, where $m$ may depend on the context. Finally, we put
	\[ \mathcal{O}_{n}(\R)=\{Q\in M_{n}^{+}(\R):Q\1=0\}.
	\]
	
	The result we shall be using from \cite{Linial 2} is the following one.
	
	\begin{theorem}\label{Linial London Rabinovich}
		Let $(X,d_{X})=(\{x_{1},\dots,x_{n}\},d_{X})$ be a finite semi-metric space and $C\geq 1$. Then $X$ admits a bi-lipschitz embedding into $\R^{n}$ with distortion at most $C$ if and only if
		\[ \sum_{q_{ij}>0}d_{X}(x_{i},x_{j})^{2}q_{ij}+C^{2}\sum_{q_{ij}<0}d_{X}(x_{i},x_{j})^{2}q_{ij}\leq 0
		\]
		for all $Q=(q_{ij})_{i,j=1}^{n}\in\mathcal{O}_{n}(\R)$.
	\end{theorem}
	
	It is worth remarking at this time that a simple consequence of the above theorem is the following expression for the Euclidean distortion of a finite semi-metric space.
	
	\begin{theorem}
		Let $(X,d_{X})=(\{x_{1},\dots,x_{n}\},d_{X})$ be a finite semi-metric space and $C\geq 1$. Then
		\[ c_{2}(X,d_{X})^{2}=\max\bigg\{\frac{\sum_{q_{ij}>0}d_{X}(x_{i},x_{j})^{2}q_{ij}}{-\sum_{q_{ij}<0}d_{X}(x_{i},x_{j})^{2}q_{ij}}:Q\in\mathcal{O}_{n}(\R), Q\neq 0\bigg\}.
		\]
	\end{theorem}
	
	Our starting point for constructing an analogue of $p$-negative type that is compatible with bi-lipschitz embeddings is Theorem \ref{Linial London Rabinovich}. Indeed, we are now in a position to provide our definition of distorted $p$-negative type, which is the main object of study in this article.
	
	\begin{definition}\label{Distorted pneg type def}
		Let $(X,d_{X})$ be a semi-metric space, $0\leq p<\infty$ and $1\leq C<\infty$. Then $X$ is said to have $p$-negative type with distortion $C$ if
		\[ \sum_{q_{ij}>0}d_{X}(x_{i},x_{j})^{p}q_{ij}+C^{2}\sum_{q_{ij}<0}d_{X}(x_{i},x_{j})^{p}q_{ij}\leq 0
		\]
		for all distinct $x_{1},\dots,x_{n}\in X$, $Q=(q_{ij})_{i,j=1}^{n}\in\mathcal{O}_{n}(\R)$ and $n\geq 2$. Furthermore, $X$ is said to have strict $p$-negative type with distortion $C$ if equality holds only for $Q=0$.
	\end{definition}
	
	In Section \ref{Preliminary Results} we prove that with the above definition of distorted $p$-negative type Theorem \ref{Linial London Rabinovich} allows us to mimic the isometric embedding property of the usual $p$-negative type in the bi-lipschitz setting. That is, a semi-metric space $(X,d_{X})$ has $p$-negative type with distortion $C$ if and only if $(X,d_{X}^{p/2})$ embeds in a Hilbert space with distortion at most $C$. We also show that distorted $p$-negative type is really a generalisation of the usual $p$-negative type. That is, for any $0\leq p<\infty$, a semi-metric space $(X,d_{X})$ has $p$-negative type with distortion $C=1$ if and only if it has $p$-negative type (in the sense of Definition \ref{Negative Type Def}). We also provide an analogue of the nesting result for the usual $p$-negative type (see the second statement in Theorem \ref{Basic Negative Type Results} above) for distorted $p$-negative type.
	
	In Section \ref{Strictness etc} we study the concepts of strict distorted $p$-negative type through the lens of polygonal equalities. The usual definition of a polygonal equality is due to Li and Weston in \cite{Li 1}, where they define a $p$-polygonal equality to be an equality of the form
		\[ \sum_{i,j=1}^{n}d_{X}(x_{i},x_{j})^{p}\xi_{i}\xi_{j}=0
		\]
	for some $n\geq2$, distinct $x_{1},\dots,x_{n}\in X$ and $\xi_{1},\dots,\xi_{n}\in\R$ with $\sum_{i=1}^{n}\xi_{i}=0$.
	
	It was proved in \cite{Li 1} that it is possible for an infinite semi-metric space $X$ to have either strict $\wp_{X}$-negative type or nonstrict $\wp_{X}$-negative type. However, if $X$ is a finite semi-metric space it was also proved in \cite{Li 1} that $X$ must have nonstrict $\wp_{X}$-negative type (provided that $\wp_{X}<\infty$). That is, when $|X|<\infty$ and $\wp_{X}<\infty$, $X$ always admits a non-trivial $\wp_{X}$-polygonal equality. This also means that if $X$ is a finite semi-metric space then $X$ has strict $p$-negative type if and only if $p<\wp_{X}$.
	
	In light of Definition \ref{Distorted pneg type def}, in the distorted setting we must define our polygonal equalities slightly differently. Here we will define a $p$-polygonal equality with distortion $C$ to be an equality of the form
		\[ \sum_{q_{ij}>0}d_{X}(x_{i},x_{j})^{p}q_{ij}+C^{2}\sum_{q_{ij}<0}d_{X}(x_{i},x_{j})^{p}q_{ij}=0
		\]
	for some distinct $x_{1},\dots,x_{n}\in X$, $Q=(q_{ij})_{i,j=1}^{n}\in\mathcal{O}_{n}(\R)$ and $n\geq 2$. The connection between our new definition of polygonal equalities and the definition due to Li and Weston is explained below in Section \ref{Strictness etc}. It suffices to say here that in Section \ref{Strictness etc} we prove that for all $p\geq\wp_{X}$, a finite metric space admits a non-trivial $p$-polygonal equality with distortion $c_{2}(X,d_{X})^{p/2}$ (see Corollary \ref{Existence Of Poly Eq}). Using this we are able to conclude that a finite semi-metric space $X$ has strict $p$-negative type with distortion $C$ if and only if $p<\wp_{X}$ or $C>c_{2}(X,d_{X}^{p/2})$ (see Corollary \ref{Distorted Strict Characterisation}).
	
	Finally, in Section \ref{Examples} we provide explicit examples of optimal distortion Euclidean embeddings of powers of the bipartite graph $K_{m,n}$ and the Hamming cube $H_{n}$. In doing so we are also able to provide explicit examples of distorted polygonal equalities for these spaces as well as determine for which values of $p$ and $C$ these spaces have (strict) $p$-negative type with distortion $C$.
	
	\section{Distorted $p$-negative Type}\label{Preliminary Results}
	
	Our first point of call is to show that when $C=1$, the definition of (strict) $p$-negative type with distortion $C$ coincides with usual definition of (strict) $p$-negative type (see Definition \ref{Negative Type Def}).
	
	\begin{proposition}\label{C=1 Equivalence}
		Let $(X,d_{X})$ be a metric space and $p\geq 0$. Then the following are true.
		\begin{enumerate}
			\item $X$ has $p$-negative type if and only if $X$ has $p$-negative type with distortion $1$.
			\item $X$ has strict $p$-negative type if and only if $X$ has strict $p$-negative type with distortion $1$.
		\end{enumerate}
	\end{proposition}
	\begin{proof}
		We will prove only the second statement, since the proof of the first statement is more or less identical. First suppose that $X$ has strict $p$-negative type with distortion $1$. By definition, this means that
		\[ \sum_{i,j=1}^{n}d_{X}(x_{i},x_{j})^{p}q_{ij}<0
		\]
		for all distinct $x_{1},\dots,x_{n}\in X$, all nonzero $Q=(q_{ij})_{i,j=1}^{n}\in\mathcal{O}_{n}(\R)$ and $n\geq 2$. Now, suppose that $n\geq 2$, $x_{1},\dots,x_{n}\in X$ are distinct and that $\xi_{1},\dots,\xi_{n}\in\R$ (not all zero) with $\sum_{i=1}^{n}\xi_{i}=0$. Setting $Q=(\xi_{i}\xi_{j})_{i,j=1}^{n}$ it is a simple matter to check that $Q\neq 0$ and  $Q\in\mathcal{O}_{n}(R)$. Hence, applying the above inequality to this particular choice of $Q$ gives that
		\[ \sum_{i,j=1}^{n}d_{X}(x_{i},x_{j})^{p}\xi_{i}\xi_{j}=\sum_{i,j=1}^{n}d_{X}(x_{i},x_{j})^{p}q_{ij}<0
		\]
		which shows that $X$ has strict $p$-negative type. Conversely, suppose that $X$ has strict $p$-negative type. Take distinct $x_{1},\dots,x_{n}\in X$ and  a nonzero $Q=(q_{ij})_{i,j=1}^{n}\in\mathcal{O}_{n}(\R)$ with $\rank(Q)=r$. Since $Q$ is positive semi-definite with $Q\1=0$, we may write it as $Q=\sum_{k=1}^{r}R_{k}$ where each $R_{k}$ is positive semi-definite with $\rank(R_{k})=1$ and $R_{k}\1=0$. By basic linear algebra, for each $1\leq k\leq r$ we can find $\xi_{1}^{(k)},\dots,\xi_{n}^{(k)}\in\R$ (not all zero) with $\sum_{i=1}^{n}\xi_{i}^{(k)}=0$ such that $R_{k}=(\xi_{i}^{(k)}\xi_{j}^{(k)})_{i,j=1}^{n}$. Putting this all togther, we have that
		\[ q_{ij}=\sum_{k=1}^{r}\xi_{i}^{(k)}\xi_{j}^{(k)}
		\]
		for all $1\leq i,j\leq n$. Hence
		\[ \sum_{i,j=1}^{n}d_{X}(x_{i},x_{j})^{p}q_{ij}=\sum_{k=1}^{r}\sum_{i,j=1}^{n}d_{X}(x_{i},x_{j})^{p}\xi_{i}^{(k)}\xi_{j}^{(k)}<0
		\]
		which shows that $X$ has strict $p$-negative type with distortion $1$.
	\end{proof}
	
	Next we remark that the bi-lipschitz embedding theorem of Linial, London and Rabinovich \cite{Linial 2} can easily be extended to deal with infinite metric spaces also. This is due to the following classical result which states that Hilbertian distortion is finitely determined. A proof of the following proposition can be found in \cite{Wells 1}, for example.
	
	\begin{proposition}\label{Local Property}
		Let $(X,d_{X})$ be a semi-metric space and $1\leq C<\infty$. Then $(X,d_{X})$ embeds in a Hilbert space with distortion at most $C$ if and only if every finite subset $Y\subseteq X$ embeds in $\ell^{2}$ with distortion at most $C$.
	\end{proposition}
	
	Combined with Theorem \ref{Linial London Rabinovich}, this allows us to say the following.
	
	\begin{theorem}\label{Infinte Distorted Embedding Theorem}
		Let $(X,d_{X})$ be a semi-metric space, $0<p<\infty$ and $1\leq C<\infty$. Then $X$ has $p$-negative type with distortion $C$ if and only if $(X,d_{X}^{p/2})$ embeds in a Hilbert space with distortion at most $C$.
	\end{theorem}
	
	It is worth noting that for finite semi-metric spaces the definition of distorted $p$-negative type simplifies somewhat, since there is no need to vary over all distinct $x_{1},\dots,x_{n}\in X$. This is a direct consequence of Theorem \ref{Linial London Rabinovich} and Theorem \ref{Infinte Distorted Embedding Theorem}.
	
	\begin{theorem}\label{Finite Distorted Embedding Theorem}
		Let $(X,d_{X})=(\{x_{1},\dots,x_{n}\},d_{X})$ be a finite semi-metric space, $0\leq p<\infty$ and $1\leq C<\infty$. Then the following statements are equivalent in pairs.
		\begin{enumerate}
			\item $X$ has $p$-negative type with distortion $C$.
			\item $(X,d_{X}^{p/2})$ embeds in $\R^{n}$ with distortion at most $C$.
			\item The inequality
			\[ \sum_{q_{ij}>0}d_{X}(x_{i},x_{j})^{p}q_{ij}+C^{2}\sum_{q_{ij}<0}d_{X}(x_{i},x_{j})^{p}q_{ij}\leq 0
			\]
			holds for all $Q=(q_{ij})_{i,j=1}^{n}\in\mathcal{O}_{n}(\R)$.
		\end{enumerate}
	\end{theorem}
	
	Again, an extremely useful corollary of the above theorem is the following formula for the Euclidean distortion of a finite semi-metric space.
	
	\begin{corollary}\label{Euclidean Distortion Formula}
		Let $(X,d_{X})=(\{x_{1},\dots,x_{n}\},d_{X})$ be a semi-metric space and $0<p<\infty$. Then
		\[ c_{2}(X,d_{X}^{p/2})^{2}=\max\bigg\{\frac{\sum_{q_{ij}>0}d_{X}(x_{i},x_{j})^{p}q_{ij}}{-\sum_{q_{ij}<0}d_{X}(x_{i},x_{j})^{p}q_{ij}}:Q\in\mathcal{O}_{n}(\R), Q\neq 0\bigg\}.
		\]
	\end{corollary}

	We now show that distorted $p$-negative type satisfies a simple nesting result.

	\begin{proposition}\label{C Nesting}
		Let $(X,d_{X})$ be a semi-metric space, $0<p<\infty$ and $1\leq C_{1}<C_{2}<\infty$. If $X$ has $p$-negative type with distortion $C_{1}$ then $X$ has strict $p$-negative type with distortion $C_{2}$.
	\end{proposition}
	\begin{proof}
		Take $n\geq 2$, distinct $x_{1},\dots,x_{n}\in X$ and $Q=(q_{ij})_{i,j=1}^{n}\in\mathcal{O}_{n}(\R)$ with $Q\neq 0$. Now, since $Q\in M_{n}^{+}(\R)$, note that $\langle Q\xi,\xi\rangle=0$ if and only if $Q\xi=0$. One direction is obvious. For the other, use the existence of the square root $Q^{1/2}$. Then $\langle Q\xi,\xi\rangle=0$ implies that $\|Q^{1/2}\xi\|_{2}^{2}=\langle Q\xi,\xi\rangle=0$ and so $Q^{1/2}\xi=0$, and thus $Q\xi=Q^{1/2}(Q^{1/2}\xi)=0$. Also, by linearity, since $Q\neq 0$ there exists some $1\leq k\leq n$ such that $Qe_{k}\neq 0$, and so $\langle Qe_{k},e_{k}\rangle>0$. Since $Q\mathbbm{1}=0$ this means there exists some $i\neq j$ such that $q_{ij}<0$ and since $i\neq j$ we also have that $q_{ij}d_{X}(x_{i},x_{j})^{p}<0$. Hence
		\[ C_{2}^{2}\sum_{q_{ij}<0}d_{X}(x_{i},x_{j})^{p}q_{ij}<C_{1}^{2}\sum_{q_{ij}<0}d_{X}(x_{i},x_{j})^{p}q_{ij}.
		\]
		But then
		\begin{align*}
			&\sum_{q_{ij}>0}d_{X}(x_{i},x_{j})^{p}q_{ij}+C_{2}^{2}\sum_{q_{ij}<0}d_{X}(x_{i},x_{j})^{p}q_{ij}
			\\&\quad\quad<\sum_{q_{ij}>0}d_{X}(x_{i},x_{j})^{p}q_{ij}+C_{1}^{2}\sum_{q_{ij}<0}d_{X}(x_{i},x_{j})^{p}q_{ij}
			\\&\quad\quad\leq 0
		\end{align*}
		which shows that $X$ has strict $p$-negative type with distortion $C_{2}$.
	\end{proof}

	We may also bootstrap the nesting property from Theorem \ref{Basic Negative Type Results} to obtain the following analogous nesting result for distorted $p$-negative type.
	
	\begin{theorem}\label{Strict Nesting}
		Let $(X,d_{X})$ be a semi-metric space, $0<q<p<\infty$ and $1\leq C<\infty$. If $X$ has $p$-negative type with distortion $C$ then $X$ has strict $q$-negative type with distortion $C^{q/p}$.
	\end{theorem}
	\begin{proof}
		Since $X$ has $p$-negative type with distortion $C$, Theorem \ref{Infinte Distorted Embedding Theorem} gives that $(X,d_{X}^{p/2})$ embeds into some Hilbert space $(H,\|\cdot\|_{2})$ with distortion at most $C$. That is, there exists a map $\phi:X\rightarrow H$ such that
			\[ d_{X}(x,y)^{p/2}\leq\|\phi(x)-\phi(y)\|_{2}\leq C\,d_{X}(x,y)^{p/2}
			\]
		for all $x,y\in X$. Take $n\geq 2$, distinct $x_{1},\dots,x_{n}\in X$ and $Q=(q_{ij})_{i,j=1}^{n}\in\mathcal{O}_{n}(\R)$, with $Q\neq 0$. Since $0<2q/p<2$, we have that $H$ has strict $2q/p$-negative type (see Theorem \ref{Basic Negative Type Results}), and so
			\[ \sum_{i,j=1}^{n}\|\phi(x_{i})-\phi(x_{j})\|^{2q/p}_{2}q_{ij}<0.
			\]
		Putting this all together, we have that
			\[ \sum_{q_{ij}>0}d_{X}(x_{i},x_{j})^{q}q_{ij}+C^{2q/p}\sum_{q_{ij}<0}d_{X}(x_{i},x_{j})^{q}q_{ij}\leq\sum_{i,j=1}^{n}\|\phi(x_{i})-\phi(x_{j})\|^{2q/p}_{2}q_{ij}<0
			\]
		and so we are done.
	\end{proof}
	
	\section{Strictness and Polygonal Equalities}\label{Strictness etc}
	
	A very large part of the recent research effort into $p$-negative type has revolved around the notions of strict $p$-negative type and polygonal equalities (see for example results in \cite{Kelleher 1, Li 1, Doust 1, Faver 1, Doust 2, Murugan 1}). In this section, we derive some basic results pertaining to these concepts in the distorted setting, including proving the existence of certain non-trivial polygonal equalities.
	
	Let us start by defining what we mean by a polygonal equality here. In the distorted setting, Theorem \ref{Infinte Distorted Embedding Theorem} prompts us to define our polygonal equalities slightly differently. When we talk about polygonal equalities from now on, we will use the following (new) definition.
	
	\begin{definition}\label{Poly Eq Def}
		Let $(X,d_{X})$ be a semi-metric space, $0\leq p<\infty$ and $1\leq C<\infty$. A $p$-polygonal equality with distortion $C$ (or a $C$-distorted $p$-polygonal equality) in $X$ is an equality of the form
		\[ \sum_{q_{ij}>0}d_{X}(x_{i},x_{j})^{p}q_{ij}+C^{2}\sum_{q_{ij}<0}d_{X}(x_{i},x_{j})^{p}q_{ij}=0
		\]
		for some distinct $x_{1},\dots,x_{n}\in X$, $Q=(q_{ij})_{i,j=1}^{n}\in\mathcal{O}_{n}(\R)$ and $n\geq 2$. Such an equality is said to be non-trivial if $Q\neq 0$. Also, the rank of such a polygonal equality is defined to be the rank of the matrix $Q$.
	\end{definition}
	
	It follows immediately from the above definition and the definition of strict $p$-negative type with distortion (Definition \ref{Distorted pneg type def}) that a semi-metric space $X$ has strict $p$-negative type with distortion $C$ if and only if it has $p$-negative type with distortion $C$ and $X$ admits no non-trivial $p$-polygonal equalities with distortion $C$.
	
	At this point, we must stop and justify our use of the terminology `polygonal equality'. Indeed the above definition in its current form looks nothing like the usual definition of polygonal equalities as they have appeared in the literature at this point in time. We now show the connection between our definition and the usual one.
	
	Let $(X,d_{X})$ be a semi-metric space, $x_{1},\dots,x_{n}\in X$ be distinct, $p\geq0$ and $Q=(q_{ij})_{i,j=1}^{n}\in\mathcal{O}_{n}(\R)$. By our definition a $p$-polygonal equality with distortion $1$ is an equality of the form
	\[ \sum_{i,j=1}^{n}d_{X}(x_{i},x_{j})^{p}q_{ij}=0.
	\]
	Now, as in the proof of Proposition \ref{C=1 Equivalence}, we have that that $\rank(Q)\leq1$ if and only if there exist $\xi_{1},\dots,\xi_{n}\in\R$ with $\sum_{i=1}^{n}\xi_{i}=0$ such that $q_{ij}=\xi_{i}\xi_{j}$, for all $1\leq i,j\leq n$. So, in this case we actually have that
	\[ \sum_{i,j=1}^{n}d_{X}(x_{i},x_{j})\xi_{i}\xi_{j}=0.
	\]
	Readers familiar with the theory of $p$-negative type will recognise this as one of the more standard definitions of a polygonal equality, from say \cite{Kelleher 1}. To summarise, our definition of rank $1$ polygonal equalities with distortion $1$ is equivalent to the standard definition of a polygonal equality from the isometric theory of $p$-negative type.
	
	An invaluable tool in the study of strict $p$-negative type and polygonal equalities in the isometric setting is the $p$-negative type gap function. This was originally defined by Doust and Weston in \cite{Doust 1} and studied extensively in \cite{Doust 1, Doust 2, Li 1, Wolf 1, Wolf 2}.
	
	Here we introduce an analogue of the $p$-negative type gap in the distorted setting. For what follows, if $n\geq 1$ and $A=(a_{ij})_{i,j=1}^{n}\in M_{n}(\R)$, we will use the notation
		\[ \pos(A)=\sum_{a_{ij}>0}a_{ij}.
		\]
	\begin{definition}\label{Delta Definition}
		Let $(X,d_{X})=(\{x_{1},\dots,x_{n}\},d_{X})$ be a finite semi-metric space. The \textit{distorted type gap function} for $X$ is defined to be the function $\Delta_{X}:[0,\infty)\times[1,\infty)\rightarrow\R$ given by
		\[ \Delta_{X}(p,C)=\inf_{\substack{Q\in\mathcal{O}_{n}(\R) \\ \pos(Q)=1}}-C^{2}\sum_{q_{ij}<0}d_{X}(x_{i},x_{j})^{p}q_{ij}-\sum_{q_{ij}>0}d_{X}(x_{i},x_{j})^{p}q_{ij}
		\]
		for all $0\leq p<\infty$ and $1\leq C<\infty$.
	\end{definition}
	
	First we remark that $\Delta_{X}(p,C)$ is always finite. To see this let $\mathcal{Q}$ denote the set of all $Q\in\mathcal{O}_{n}(\R)$ with $\pos(Q)=1$ and topologize $\mathcal{Q}$ with the pointwise topology (i.e. $Q_{k}\rightarrow Q$ if and only if $Q_{k}$ converges to $Q$ entrywise). Note that since $M_{n}(\R)$ is finite-dimensional this coincides with the restriction of the unique norm topology on $M_{n}(\R)$ to $\mathcal{Q}$. It is a simple matter to check that with this topology $\mathcal{Q}$ is compact (it is a closed and bounded subset of $M_{n}(\R)$). Also, define $f:[0,\infty)\times[1,\infty)\times\mathcal{Q}\rightarrow\R$ by
	\[ f(p,C,Q)=-C^{2}\sum_{q_{ij}<0}d_{X}(x_{i},x_{j})^{p}q_{ij}-\sum_{q_{ij}>0}d_{X}(x_{i},x_{j})^{p}q_{ij}.
	\]
	Then $f$ is continuous and also
	\[ \Delta_{X}(p,C)=\inf_{Q\in\mathcal{Q}}f(p,C,Q).
	\]
	So, since $f$ is continuous and $\mathcal{Q}$ is compact it follows that $\Delta_{X}(p,C)$ is always finite and that the above infimum is always attained for some $Q\in\mathcal{Q}$.
	
	The fact that $\Delta_{X}(p,C)$ is the infimum over a compact set also implies that it must be continuous. Here we will use the following result from elementary analysis. The proof of this result is left as an exercise for the reader.
	
	\begin{proposition}\label{Inf Continuity}
		Let $(Y,\tau_{Y})$ be a topological space and let $(Z,\tau_{Z})$ be a compact topological space. Suppose that $f:Y\times Z\rightarrow\R$ is continuous (with respect to the product topology on $Y\times Z$) and define $g:Y\rightarrow\R$ by
		\[ g(y)=\inf_{z\in Z}f(y,z)
		\]
		for all $y\in Y$. Then $g$ is continuous.
	\end{proposition}
	
	\begin{corollary}\label{Continuity Results}
		Let $(X,d_{X})=(\{x_{1},\dots,x_{n}\},d_{X})$ be a finite semi-metric space. Then $\Delta_{X}:[0,\infty)\times[1,\infty)\rightarrow\R$ is continuous (with respect to the product topology on $[0,\infty)\times[1,\infty)$.
	\end{corollary}
	
	It follows immediately from the above definition (and the fact that one may rescale appropriately) that if $(X,d_{X})=(\{x_{1},\dots,x_{n}\},d_{X})$ is a finite semi-metric space, $0\leq p<\infty$ and $1\leq C<\infty$ then $X$ has $p$-negative type with distortion $C$ if and only if $\Delta_{X}(p,C)\geq 0$. More importantly however is the following refinement of this property when dealing with strict distorted $p$-negative type. That is, $\Delta_{X}$ has the following property.
	
	\begin{theorem}\label{Correct Delta}
		Let $(X,d_{X})=(\{x_{1},\dots,x_{n}\},d_{X})$ be a finite semi-metric space, $0\leq p<\infty$ and $1\leq C<\infty$. Then $X$ has strict $p$-negative type with distortion $C$ if and only if $\Delta_{X}(p,C)>0$.
	\end{theorem}
	\begin{proof}
		First suppose that $\Delta_{X}(p,C)>0$. So, take $Q=(q_{ij})_{i,j=1}^{n}\in\mathcal{O}_{n}(\R)$ with $Q\neq 0$. Then, since $\pos(Q)\neq 0$ one may define $R=(r_{ij})_{i,j=1}^{n}$ by $R=\pos(Q)^{-1}Q$. Of course, one then has that $R\in\mathcal{O}_{n}(\R)$ with $\pos(R)=\pos(Q)\pos(Q)^{-1}=1$. Hence
		\begin{align*}
			&-\frac{1}{\pos(Q)}\bigg(\sum_{q_{ij}>0}d_{X}(x_{i},x_{j})^{p}q_{ij}+C^{2}\sum_{q_{ij}<0}d_{X}(x_{i},x_{j})^{p}q_{ij}\bigg)
			\\&\quad\quad=-C^{2}\sum_{r_{ij}<0}d_{X}(x_{i},x_{j})^{p}r_{ij}-\sum_{r_{ij}>0}d_{X}(x_{i},x_{j})^{p}r_{ij}
			\\&\quad\quad\geq\Delta_{X}(p,C)
			\\&\quad\quad>0.
		\end{align*}
		Hence, after dividing both sides by $-\pos(Q)^{-1}$ one finds that
		\[ \sum_{q_{ij}>0}d_{X}(x_{i},x_{j})^{p}q_{ij}+C^{2}\sum_{q_{ij}<0}d_{X}(x_{i},x_{j})^{p}q_{ij}<0
		\]
		which shows that $X$ has strict $p$-negative type with distortion $C$. Conversely, suppose that $X$ has strict $p$-negative type with distortion $C$. Let us keep the notation from below Definition \ref{Delta Definition} so that
		\[ \Delta_{X}(p,C)=\inf_{Q\in\mathcal{Q}} f(p,C,Q).
		\]
		Since $f$ is continuous and $\mathcal{Q}$ is compact this infimum must be attained. That is, there exists some $Q_{0}\in\mathcal{Q}$ (that depends on both $p$ and $C$) such that
		\[ \Delta_{X}(p,C)=\inf_{Q\in\mathcal{Q}} f(p,C,Q)=f(p,C,Q_{0}).
		\]
		But since $X$ has strict $p$-negative type with distortion $C$ one has that $f(p,C,Q_{0})>0$ and hence by the above equation we conclude that $\Delta_{X}(p,C)>0$ as required.
	\end{proof}
	
	Combining this with Corollary \ref{Continuity Results} gives the following result pertaining to strict distorted $p$-negative type.
	
	\begin{corollary}\label{Small Extension}
		Let $(X,d_{X})$ be a finite semi-metric space, $0\leq p<\infty$ and $1\leq C<\infty$. Suppose that $X$ has strict $p$-negative type with distortion $C$. Then there exists some $\zeta>0$ (that depends on both $p$ and $C$) such that $X$ has strict $q$-negative type with distortion $K$ for all $q\in[p,p+\zeta]$ and $K\in[\max(C-\zeta,1),C]$.
	\end{corollary}
	
	Thus we are now able to obtain a generalisation of the fact that all finite semi-metric spaces admit a non-trivial non-distorted $\wp_{X}$-polygonal equality (recall that $\wp_{X}$ is used to denote the supremal $p$-negative type of $X$).
	
	\begin{corollary}\label{Existence Of Poly Eq}
		Let $(X,d_{X})$ be a finite semi-metric space and $p\geq\wp_{X}$. Then there exists a non-trivial $c_{2}(X,d_{X}^{p/2})$-distorted $p$-polygonal equality in $X$.
	\end{corollary}
	\begin{proof}
		For ease of notation, let us denote $c_{2}(X,d_{X}^{p/2})$ simply by $c_{2}(X^{p/2})$. By Theorem \ref{Infinte Distorted Embedding Theorem} note that $X$ has $p$-negative type with distortion $c_{2}(X^{p/2})$, for all $p\geq 0$. In particular this means that $\Delta_{X}(p,c_{2}(X^{p/2}))\geq 0$, for all $p\geq 0$. So, take $p\geq\wp_{X}$ and let us assume for a contradiction that $\Delta_{X}(p,c_{2}(X^{p/2}))>0$. Then by Theorem \ref{Correct Delta} this means that $X$ has strict $p$-negative type with distortion $c_{2}(X^{p/2})$ and hence by Corollary \ref{Small Extension} there exists some $\zeta>0$ such that $X$ also has strict $(p+\zeta)$-negative type with distortion $c_{2}(X^{p/2})$. By Theorem \ref{Infinte Distorted Embedding Theorem} this means that $(X,d_{X}^{(p+\zeta)/2})$ embeds in $\ell^{2}$ with distortion at most $c_{2}(X^{p/2})$ and hence $c_{2}(X^{(p+\zeta)/2})\leq c_{2}(X^{p/2})$. But this is impossible since the function $r\mapsto c_{2}(X^{r/2})$ is strictly increasing for $r\geq\wp_{X}$ (see Theorem \ref{Strict Nesting}). Hence it must be the case that $\Delta_{X}(p,c_{2}(X^{p/2}))=0$.
		
		Now, arguing as in the proof of Theorem \ref{Correct Delta} (and keeping the notation from that proof) there must exist some $Q_{0}\in\mathcal{Q}$ such that
		\[ f(p,c_{2}(X^{p/2}),Q_{0})=\Delta_{X}(p,c_{2}(X^{p/2}))=0.
		\]
		But this is just another way of saying that $Q_{0}$ is a $c_{2}(X^{p/2})$-distorted $p$-polygonal equality in $X$. Also, since $Q_{0}\in\mathcal{Q}$ we have that $\pos(Q_{0})=1$ and hence $Q_{0}\neq 0$. Hence $Q_{0}$ is a non-trivial $c_{2}(X^{p/2})$-distorted $p$-polygonal equality in $X$ and so we are done.
	\end{proof}
	
	The above corollary enables us to classify those $p$ and $C$ for which a finite semi-metric space $X$ has strict $p$-negative type with distortion $C$.
	
	\begin{corollary}\label{Distorted Strict Characterisation}
		Let $(X,d_{X})$ be a finite semi-metric space, $0\leq p<\infty$ and $C\geq 1$. Then $X$ has strict $p$-negative type with distortion $C$ if and only if $p<\wp_{X}$ or $C>c_{2}(X,d_{X}^{p/2})$.
	\end{corollary}
	\begin{proof}
		First suppose that $p<\wp_{X}$. By the definition of $\wp_{X}$ we have that $X$ has $\wp_{X}$-negative type with distortion $1$ and hence by Theorem \ref{Strict Nesting} $X$ also has strict $p$-negative type with distortion $1$. But then by Proposition \ref{C Nesting} we also have that $X$ has strict $p$-negative type with distortion $C$. Now, suppose instead that $C>c_{2}(X,d_{X}^{p/2})$. Then Theorem \ref{Infinte Distorted Embedding Theorem} gives that $X$ has $p$-negative type with distortion $c_{2}(X,d_{X}^{p/2})$ and so Proposition \ref{C Nesting} again shows that $X$ has strict $p$-negative type with distortion $C$.
		
		Conversely, suppose that $X$ has strict $p$-negative type with distortion $C$ and that $p\geq\wp_{X}$. Then Theorem \ref{Infinte Distorted Embedding Theorem} implies that $C\geq c_{2}(X,d_{X}^{p/2})$. But by Corollary \ref{Existence Of Poly Eq} we know that $X$ has nonstrict $p$-negative type with distortion $c_{2}(X,d_{X}^{p/2})$ and so it must be the case that $C>c_{2}(X,d_{X}^{p/2})$.
	\end{proof}
	
	\section{Examples} \label{Examples}
	
	In this section we provide two examples of semi-metric spaces and their values of (strict) distorted $p$-negative type, as well as some of their non-trivial polygonal equalities. The first example we look at is that of the bipartite graph $K_{m,n}$, where the results here are entirely new. For the second example, we make use of results from Linial and Magen \cite{Linial 1} to compute the values of distorted $p$-negative type of the Hamming cube $H_{n}$.
	
	Throughout this section we will use the following standard notation when computing the distortion of a given embedding.
	
	\begin{definition}\label{Expansion Contraction Definition}
		Let $(X,d_{X})$ and $(Y,d_{Y})$ be semi-metric spaces, and suppose that $f:X\rightarrow Y$.
		\begin{enumerate}
			\item The \textit{contraction} of $f$ is defined to be
			\[ \cont(f)=\sup_{\substack{x,y\in X \\ x\neq y}}\frac{d_{X}(x,y)}{d_{Y}(f(x),f(y))}.
			\]
			\item The \textit{expansion} of $f$ is defined to be
			\[ \expan(f)=\sup_{\substack{x,y\in X \\ x\neq y}}\frac{d_{Y}(f(x),f(y))}{d_{X}(x,y)}.
			\]
		\end{enumerate}
	\end{definition}
	
	It is a simple matter to check that using this notation one has that if $(X,d_{X})$ and $(Y,d_{Y})$ are semi-metric spaces and $f:X\rightarrow Y$ is a bi-lipschitz embedding then the distortion of $f$ is given by
	\[ \dist(f)=\cont(f)\times\expan(f).
	\]
	
	\subsection{The Bipartite Graph $K_{m,n}$}
	
	The first example that we will concern ourselves with will be the bipartite graph $K_{m,n}$ equipped with its graph metric. When we say let $(X,d_{X})$ be the bipartite graph $K_{m,n}$ we mean that $X$ is the space $X=\{u_{1},\dots,u_{m},v_{1},\dots,v_{n}\}$ with metric $d_{X}$ defined by $d_{X}(u_{i},u_{j})=d_{X}(v_{k},v_{l})=2$, for all $1\leq i,j\leq m$ and $1\leq k,l\leq n$ with $i\neq j$ and $k\neq l$, and also $d_{X}(u_{i},v_{j})=1$ for all $1\leq i\leq m$ and $1\leq j\leq n$.
	
	To properly describe the optimal embeddings of powers of $K_{m,n}$ into Hilbert space we first need to understand how the complete graph $K_{n}$ can be isometrically embedded into Hilbert space. For what follows when we say let $(X,d_{X})$ be the complete graph $K_{n}$ we mean that $X=\{u_{1},\dots,u_{n}\}$ with metric $d_{X}$ such that $d_{X}(u_{i},u_{j})=1$, for all $1\leq i,j\leq n$ with $i\neq j$.
	
	It is a simple matter to construct an isometric embedding of $K_{n}$ into $\R^{n}$. Indeed, one may simply take $u_{i}\mapsto e_{i}/\sqrt{2}$ for all $1\leq i\leq n$ where $e_{1},\dots,e_{n}$ are the standard basis vectors. However, since $K_{n}$ is an $n$ point metric space it must therefore be possible to isometrically embed $K_{n}$ into\footnote{It can also be shown that $K_{n}$ cannot be isometrically embedded into $\R^{r}$ for any $r<n-1$. Indeed, since $K_{n}$ is an ultrametric space it has strict $2$-negative type and hence it must embed isometrically into $\R^{n-1}$ as an affinely independent set. For such results see \cite{Faver 1}.} $\R^{n-1}$. While the problem of writing an explicit formula for an isometric embedding of $K_{n}$ into $\R^{n-1}$ is not a conceptually challenging one it is rather tedious. Let
	\[ c_{n}=\frac{\sqrt{2}(1+\sqrt{n})}{2(n-1)}
	\]
	and 
	\[ C_{n}=\frac{1}{n}\bigg(c_{n}+\frac{1}{\sqrt{2}}\bigg)\mathbbm{1}
	\]
	where here $\mathbbm{1}$ denotes the vector in $\R^{n-1}$ all of whose entries are $1$. Then define $\phi:K_{n}\rightarrow\R^{n-1}$ by $\phi(u_{i})=e_{i}/\sqrt{2}-C_{n}$ for all $1\leq i\leq n-1$ and $\phi(u_{n})=c_{n}\mathbbm{1}-C_{n}$. It is a simple yet tedious task to check that $\phi$ is an isometric embedding of $K_{n}$ into $\R^{n-1}$. It is also simple to check that $\|\phi(u_{i})\|_{2}=(1-1/n)^{1/2}/2^{1/2}$ for all $1\leq i\leq n$. We will refer to this particular embedding as the standard embedding of $K_{n}$ into $\R^{n-1}$.
	
	We now move on to the problem of desribing optimal embeddings of powers of $K_{m,n}$. In proving the optimality of our embeddings we will require a matrix $Q$ that will serve as a distorted polygonal equality for $K_{m,n}$.
	
	\begin{lemma}\label{Kmn Poly Eq}
		Define $Q=(q_{ij})_{i,j=1}^{m+n}\in M_{m+n}(\R)$ by
		\[ q_{ij}=\begin{cases} \frac{1}{m^{2}},&\mbox{ if }1\leq i,j\leq m, \\ \frac{1}{n^{2}},&\mbox{ if }m+1\leq i,j\leq m+n, \\ -\frac{1}{mn},&\mbox{ if }1\leq i\leq m,m+1\leq j\leq m+n, \\ -\frac{1}{mn},&\mbox{ if }m+1\leq i\leq m+n,1\leq j\leq m.\end{cases}
		\]
		Then $Q\in\mathcal{O}_{m+n}(\R)$.
	\end{lemma}
	\begin{proof}
		Define $\xi_{1},\dots,\xi_{m+n}\in\R$ by
		\[ \xi_{i}=\begin{cases} \frac{1}{m},&\mbox{ if }1\leq i\leq m, \\ -\frac{1}{n},&\mbox{ if }m+1\leq i\leq m+n.\end{cases}
		\]
		Then $\sum_{i=1}^{m+n}\xi_{i}=0$ and $Q=(\xi_{i}\xi_{j})_{i,j=1}^{m+n}$. As in the proof of Proposition \ref{C=1 Equivalence}, it now follows that $Q\in\mathcal{O}_{m+n}(\R)$.
	\end{proof}
	
	In what follows if $r\geq 1$ then we shall use $0_{r}$ to denote the zero vector in $\R^{r}$ (i.e. the vector in $\R^{r}$ all of whose coordinates are zero). Also, for ease of notation we shall set
	\[ \wp_{m,n}=\log_{2}\bigg(\frac{2mn}{2mn-m-n}\bigg).
	\]
	
	\begin{theorem}\label{Kmn Embedding}
		Let $(X,d_{X})$ be the bipartite graph $K_{m,n}$, where $m,n\geq 1$ (not both $1$) and let $\wp=\wp_{m,n}$. Also, let $x_{1},\dots,x_{m}\in\R^{m-1}$ be the image of the standard embedding of $K_{m}$ into $\R^{m-1}$ and let $y_{1},\dots,y_{n}\in\R^{n-1}$ be the image of the standard embedding of $K_{n}$ into $\R^{n-1}$ (see the comments above Lemma \ref{Kmn Poly Eq}). Then for $p\geq\wp$ the map $\phi:(X,d_{X}^{p/2})\rightarrow\R^{m-1}\oplus\R^{n-1}=\R^{m+n-2}$ defined by
		\begin{align*}
			\phi(u_{i})&=(x_{i},0_{n-1}),\,\forall 1\leq i\leq m,
			\\\phi(v_{j})&=(0_{m-1},y_{j}),\,\forall 1\leq j\leq n,
		\end{align*}
		has $\dist(\phi)=c_{2}(X,d_{X}^{p/2})$. Consequently $\wp_{X}=\wp$ and
		\[ c_{2}(X,d_{X}^{p/2})=\begin{cases} 1,&\mbox{ if }0\leq p\leq\wp, \\ 2^{p/2}\bigg(1-\frac{1}{2}\bigg(\frac{1}{m}+\frac{1}{n}\bigg)\bigg)^{1/2},&\mbox{ if }\wp\leq p<\infty.\end{cases}
		\]
	\end{theorem}
	\begin{proof}
		It follows immediately from the definition of $\phi$ that
		\[ \|\phi(u_{i})-\phi(u_{j})\|_{2}=1=\|\phi(v_{k})-\phi(v_{l})\|_{2}
		\]
		for all $1\leq i\neq j\leq m$, $1\leq k\neq l\leq n$. Now take $1\leq i\leq m$ and $1\leq j\leq n$. By what was said above about the standard embeddings of $K_{r}$ into $\R^{r-1}$ we have that $\|\phi(u_{i})\|_{2}=(1-1/m)^{1/2}/2^{1/2}$ and $\|\phi(v_{j})\|_{2}=(1-1/n)^{1/2}/2^{1/2}$. Also, since $\phi(u_{i})$ and $\phi(v_{j})$ are clearly orthogonal we have that
		\begin{align*}
			\|\phi(u_{i})-\phi(v_{j})\|_{2}^{2}&=\|\phi(u_{i})\|_{2}^{2}+\|\phi(v_{j})\|_{2}^{2}
			\\&=\frac{1}{2}\bigg(1-\frac{1}{m}\bigg)+\frac{1}{2}\bigg(1-\frac{1}{n}\bigg)
			\\&=1-\frac{1}{2}\bigg(\frac{1}{m}+\frac{1}{n}\bigg).
		\end{align*}
		Hence, since we are we are thinking of $\phi$ as a map $\phi:(X,d_{X}^{p/2})\rightarrow\R^{m+n-2}$, one has that
		\begin{align*}
			\expan(\phi)&=\sup_{\substack{x,y\in X \\ x\neq y}}\frac{\|\phi(x)-\phi(y)\|_{2}}{d_{X}(x,y)^{p/2}}
			\\&=\max\bigg(\frac{1}{2^{p/2}},\bigg(1-\frac{1}{2}\bigg(\frac{1}{m}+\frac{1}{n}\bigg)\bigg)^{1/2}\bigg).
		\end{align*}
		Since $p\geq\wp=\wp_{m,n}$ it is readily checked that
		\[ \expan(\phi)=\bigg(1-\frac{1}{2}\bigg(\frac{1}{m}+\frac{1}{n}\bigg)\bigg)^{1/2}.
		\]
		Similarly, since $p\geq\wp=\wp_{m,n}$ the contraction of $\phi$ is given by
		\begin{align*}
			\cont(\phi)&=\sup_{\substack{x,y\in X \\ x\neq y}}\frac{d_{X}(x,y)^{p/2}}{\|\phi(x)-\phi(y)\|_{2}}
			\\&=\max\bigg(2^{p/2},\bigg(1-\frac{1}{2}\bigg(\frac{1}{m}+\frac{1}{n}\bigg)\bigg)^{-1/2}\bigg)
			\\&=2^{p/2}.
		\end{align*}
		Hence the distortion of $\phi:(X,d_{X}^{p/2})\rightarrow\R^{m+n-2}$ is
		\[ \dist(\phi)=\cont(\phi)\times\expan(\phi)=2^{p/2}\bigg(1-\frac{1}{2}\bigg(\frac{1}{m}+\frac{1}{n}\bigg)\bigg)^{1/2}.
		\]
		For the lower bound, let $Q=(q_{ij})_{i,j=1}^{m+n}$ be the matrix defined in Lemma \ref{Kmn Poly Eq}, which we know is in $\mathcal{O}_{m+n}(\R)$. Then taking $z_{1}=u_{1},\dots,z_{m}=u_{m},z_{m+1}=v_{1},\dots,z_{m+n}=v_{n}$, Theorem \ref{Euclidean Distortion Formula} gives that
		
		\begin{align*}
			c_{2}(X,d_{X}^{p/2})^{2}&\geq-\frac{\sum_{q_{ij}>0}d_{X}(z_{i},z_{j})^{p}q_{ij}}{\sum_{q_{ij}<0}d_{X}(z_{i},z_{j})^{p}q_{ij}}
			\\&=2^{p}\bigg(1-\frac{1}{2}\bigg(\frac{1}{m}+\frac{1}{n}\bigg)\bigg).
		\end{align*}
		Thus for $p\geq\wp$ one has that
		\[ c_{2}(X,d_{X}^{p/2})=2^{p/2}\bigg(1-\frac{1}{2}\bigg(\frac{1}{m}+\frac{1}{n}\bigg)\bigg)^{1/2}.
		\]
		Note in particular that $c_{2}(X,d_{X}^{\wp/2})=1$. The fact that $c_{2}(X,d_{X}^{p/2})=1$ for all $0\leq p<\wp$ now follows from Theorem \ref{Strict Nesting}.
	\end{proof}
	
	\begin{corollary}
		Let $(X,d_{X})$ be the bipartite graph $K_{m,n}$, where $m,n\geq 1$ (not both $1$), $\wp=\wp_{m,n}$, $0\leq p<\infty$ and $1\leq C<\infty$. Then $X$ has $p$-negative type with distortion $C$ if and only if $0\leq p\leq\wp$, or $p>\wp$ and
		\[ C\geq2^{p/2}\bigg(1-\frac{1}{2}\bigg(\frac{1}{m}+\frac{1}{n}\bigg)\bigg)^{1/2}.
		\]
	\end{corollary}
	\begin{proof}
		This is a direct consequence of the above theorem and Theorem \ref{Infinte Distorted Embedding Theorem}.
	\end{proof}
	
	\begin{corollary}
		Let $(X,d_{X})$ be the bipartite graph $K_{m,n}$, where $m,n\geq 1$ (not both $1$), $\wp=\wp_{m,n}$, $0\leq p<\infty$ and $1\leq C<\infty$. Then $X$ has strict $p$-negative type with distortion $C$ if and only if $0\leq p<\wp$, or $p\geq\wp$ and
		\[ C>2^{p/2}\bigg(1-\frac{1}{2}\bigg(\frac{1}{m}+\frac{1}{n}\bigg)\bigg)^{1/2}.
		\]
	\end{corollary}
	\begin{proof}
		This is a direct consequence of Theorem \ref{Kmn Embedding} and Corollary \ref{Distorted Strict Characterisation}.
	\end{proof}
	
	\subsection{The Hamming Cube $H_{n}$}
	
	For our second example, we study the Hamming cube $H_{n}$. For what follows, when we say let $(X,d_{X})$ be the Hamming cube $H_{n}$ we mean that $X=\{0,1\}^{n}\subseteq\R^{n}$ with metric $d_{X}$ defined by $d_{X}(x,y)=\sum_{i=1}^{n}|x_{i}-y_{i}|$, for all $x=(x_{1},\dots,x_{n})^{T},y=(y_{1},\dots,y_{n})^{T}\in H_{n}$.
	
	Once again, we start by detailing our canditate for a polygonal equality for $H_{n}$. Here it will be easier to omit mention of any explicit ordering of the elements of $X=H_{n}$ and instead use the notation $Q=(q_{x,y})_{x,y\in X}\in M_{2^{n}}(\R)$.
	
	A proof of the following lemma can be found in \cite{Linial 1}.
	
	\begin{lemma}\label{Hamming Cube Nullity}
		Let $(X,d_{X})$ be the Hamming cube $H_{n}$, where $n\geq 2$, and define $Q=(q_{x,y})_{x,y\in X}$ by
		\[ q_{x,y}=\begin{cases} n-1,&\mbox{ if }x=y, \\ -1,&\mbox{ if }d_{X}(x,y)=1, \\ 1,&\mbox{ if }d_{X}(x,y)=n, \\ 0,&\mbox{ otherwise.}\end{cases}
		\]
		Then $Q\in\mathcal{O}_{2^{n}}(\R)$.
	\end{lemma}
	
	\begin{theorem}\label{Hamming Cube Embedding}
		Let $(X,d_{X})$ be the Hamming cube $H_{n}$, where $n\geq 2$. For $p\geq 1$ define the map $\phi:(X,d_{X}^{p/2})\rightarrow\R^{n}$ by $\phi(x)=x$. Then $\phi$ is an optimal distortion embedding. Consequently, one has that
		\[ c_{2}(X,d_{X}^{p/2})=\begin{cases} 1,&\mbox{ if }0\leq p\leq 1, \\ n^{(p-1)/2} &,\mbox{ if }1\leq p<\infty.\end{cases}
		\]
	\end{theorem}
	\begin{proof}
		Note that if $x_{k},y_{k}\in\{0,1\}$ then $x_{k}-y_{k}\in\{0,\pm1\}$ and so $(x_{k}-y_{k})^{2}=|x_{k}-y_{k}|$. Hence if $x=(x_{1},\dots,x_{n})^{T},y=(y_{1},\dots,y_{n})^{T}\in X$ then
		\[ \|\phi(x)-\phi(y)\|_{2}^{2}=\sum_{k=1}^{n}(x_{k}-y_{k})^{2}=\sum_{k=1}^{n}|x_{k}-y_{k}|=d_{X}(x,y).
		\]
		Hence if $d_{X}(x,y)=k$ then $\|\phi(x)-\phi(y)\|_{2}=k^{1/2}$. Since $p\geq 1$ the expansion and contraction of $\phi$ are then given by
		\[ \expan(\phi)=\sup_{\substack{x,y\in X \\ x\neq y}}\frac{\|\phi(x)-\phi(y)\|_{2}}{d_{X}(x,y)^{p/2}}=\max_{1\leq k\leq n}\frac{k^{1/2}}{k^{p/2}}=\max_{1\leq k\leq n}k^{(1-p)/2}=1
		\]
		and
		\[ \cont(\phi)=\sup_{\substack{x,y\in X \\ x\neq y}}\frac{d_{X}(x,y)^{p/2}}{\|\phi(x)-\phi(y)\|_{2}}=\max_{1\leq k\leq n}\frac{k^{p/2}}{k^{1/2}}=\max_{1\leq k\leq n}k^{(p-1)/2}=n^{(p-1)/2}.
		\]
		Hence the distortion of $\phi$ is given by
		\[ \dist(\phi)=\cont(\phi)\times\expan(\phi)=n^{(p-1)/2}.
		\]	
		Now all that we need to do is show that this cannot be improved. Define $Q=(q_{x,y})_{x,y\in X}$ as in Lemma \ref{Hamming Cube Nullity}. Then by Corollary \ref{Euclidean Distortion Formula} we have that
		\begin{align*}
			c_{2}(X,d_{X}^{p/2})^{2}&\geq-\frac{\sum_{q_{x,y}>0}d_{X}(x,y)^{p}q_{x,y}}{\sum_{q_{x,y}<0}d_{X}(x,y)^{p}q_{x,y}}
			\\&=-\frac{\sum_{d_{X}(x,y)=n}d_{X}(x,y)^{p}q_{x,y}}{\sum_{d_{X}(x,y)=1}d_{X}(x,y)^{p}q_{x,y}}
			\\&=-\frac{2^{n}\times n^{p}\times 1}{n2^{n}\times 1^{p}\times -1}
			\\&=n^{p-1}.
		\end{align*}
		Hence if $p\geq 1$ then
		\[ c_{2}(X,d_{X}^{p/2})=n^{(p-1)/2}.
		\]
		Note in particular that $c_{2}(X,d_{X}^{1/2})=1$. The fact that $c_{2}(X,d_{X}^{p/2})=1$ for all $0\leq p<1$ now follows from Theorem \ref{Strict Nesting}.
	\end{proof}
	
	\begin{corollary}
		Let $(X,d_{X})$ be the Hamming cube $H_{n}$, where $n\geq 2$, $0\leq p<\infty$ and $1\leq C<\infty$. Then $X$ has $p$-negative type with distortion $C$ if and only if $0\leq p\leq 1$, or $p>1$ and
		\[ C\geq n^{(p-1)/2}.
		\]
	\end{corollary}
	\begin{proof}
		This is a direct consequence of the above theorem and Theorem \ref{Infinte Distorted Embedding Theorem}.
	\end{proof}
	
	\begin{corollary}
		Let $(X,d_{X})$ be the Hamming cube $H_{n}$, where $n\geq 2$, $0\leq p<\infty$ and $1\leq C<\infty$. Then $X$ has strict $p$-negative type with distortion $C$ if and only if $0\leq p<1$, or $p\geq1$ and
		\[ C>n^{(p-1)/2}.
		\]
	\end{corollary}
	\begin{proof}
		This is a direct consequence of Theorem \ref{Hamming Cube Embedding} and Corollary \ref{Distorted Strict Characterisation}.
	\end{proof}
	
	\section*{Acknowledgments}
	
	The work of the author was supported by the Research Training Program of the Department of Education and Training of the
	Australian Government. The author also wishes to thank Ian Doust for his help in reading over many drafts of this paper.
	
	\newpage
	
	\bibliographystyle{amsalpha}

\end{document}